\pdfoutput=1
\RequirePackage{ifpdf}
\ifpdf
\documentclass[pdftex]{sigma}
\else
\documentclass{sigma}
\fi

\renewcommand{\d}{{\mathrm d}}
\newcommand{\Li}{\operatorname{Li}}

\begin{document}

\newcommand{\arXivNumber}{1801.09895}

\renewcommand{\thefootnote}{}

\renewcommand{\PaperNumber}{028}

\FirstPageHeading

\ShortArticleName{One of the Odd Zeta Values from $\zeta(5)$ to $\zeta(25)$ Is Irrational}

\ArticleName{One of the Odd Zeta Values\\ from $\boldsymbol{\zeta(5)}$ to $\boldsymbol{\zeta(25)}$ Is Irrational.\\ By Elementary Means\footnote{This paper is a~contribution to the Special Issue on Orthogonal Polynomials, Special Functions and Applications (OPSFA14). The full collection is available at \href{https://www.emis.de/journals/SIGMA/OPSFA2017.html}{https://www.emis.de/journals/SIGMA/OPSFA2017.html}}}

\Author{Wadim ZUDILIN}
\AuthorNameForHeading{W.~Zudilin}

\Address{Department of Mathematics, IMAPP, Radboud University,\\
PO Box 9010, 6500~GL Nijmegen, The Netherlands}
\Email{\href{mailto:w.zudilin@math.ru.nl}{w.zudilin@math.ru.nl}}
\URLaddress{\url{http://www.math.ru.nl/~wzudilin/}}

\ArticleDates{Received January 31, 2018, in final form March 26, 2018; Published online March 29, 2018}

\Abstract{Available proofs of result of the type `at least one of the odd zeta values $\zeta(5),\zeta(7),\dots,\zeta(s)$ is irrational'
make use of the saddle-point method or of linear independence criteria, or both. These two remarkable techniques are however counted as highly non-elementary, therefore leaving the partial irrationality result inaccessible to general mathe\-matics audience in all its glory. Here we modify the original construction of linear forms in odd zeta values to produce, for the first time, an elementary proof of such a result~--- a~proof whose technical ingredients are limited to the prime number theorem and Stirling's approximation formula for the factorial.}

\Classification{11J72; 11M06; 33C20}

\renewcommand{\thefootnote}{\arabic{footnote}}
\setcounter{footnote}{0}

\section{Introduction}\label{s1}

Without touching deeply a history of the question (see \cite{Fi04} for an excellent account of this), we notice that the irrationality of the \emph{zeta values} --- values of Riemann's zeta function
\begin{gather*}
\zeta(s)=\sum_{n=1}^\infty\frac1{n^s}
\end{gather*}
at integers $s=2,3,\dots$, is known for even $s$ and also for $s=3$, while there are only partial results in this direction for odd $s\ge5$.
A starting point here has been set in the work \cite{BR01} with further development, particularly focusing on $\zeta(5)$, in~\cite{Ri02} and~\cite{Zu04}.

We fix an odd integer $s\ge7$. Our strategy is constructing two sequences of linear forms~$r_n$ and~$\hat r_n$ living in the $\mathbb Q$-space $\mathbb Q+\mathbb Q\zeta(3)+\mathbb Q\zeta(5)+\dots+\mathbb Q\zeta(s)$, for which we have a control of the common denominators $\lambda_n$ of rational coefficients and an elementary access to their asymptotic behaviour as $n\to\infty$; more importantly, the two coefficients of $\zeta(3)$ in these forms are proportional (with factor $7$), so that $7r_n-\hat r_n$ belongs to the space $\mathbb Q+\mathbb Q\zeta(5)+\dots+\mathbb Q\zeta(s)$. Finally, using $7r_n-\hat r_n>0$ and the asymptotics $\lambda_n(7r_n-\hat r_n)\to0$ as $n\to\infty$ of the linear forms
\begin{gather*}
\lambda_n(7r_n-\hat r_n)\in\mathbb Z+\mathbb Z\zeta(5)+\mathbb Z\zeta(7)+\dots+\mathbb Z\zeta(s)
\end{gather*}
when $s=25$, we conclude that it cannot happen that all the quantities $\zeta(5),\zeta(7),\dots,\allowbreak\zeta(25)$ are rational.

The original idea of using the so-called well-poised hypergeometric series to construct linear forms in zeta values of a given parity is due
to Ball and Rivoal~\cite{BR01}; our new ingredient here is using simultaneously such series and their `twists by half', for an appropriate choice of the parameters. More precisely, our hypergeometric series assume the form
\begin{gather}
r_n=\sum_{\nu=1}^\infty R_n(\nu) \qquad\text{and}\qquad \hat r_n=\sum_{\nu=1}^\infty R_n\big(\nu-\tfrac12\big),
\label{e01}
\end{gather}
where the rational-function summand $R_n(t)$ is defined as follows:
\begin{align}
R(t)=R_n(t)
&=\frac{n!^{s-5}\prod_{j=1}^n(t-j)\cdot\prod_{j=1}^n(t+n+j)\cdot2^{6n}\prod_{j=1}^{3n}\big(t-n-\frac12+j\big)} {\prod_{j=0}^n(t+j)^s}\nonumber\\
&=\frac{2^{6n}n!^{s-5}\prod_{j=0}^{6n}\big(t-n+\frac12j\big)}{\prod_{j=0}^n(t+j)^{s+1}}.\label{e02}
\end{align}

The following Sections~\ref{s2} and~\ref{s3} discuss, respectively, the arithmetic and analysis of the forms~\eqref{e01}. In Section~\ref{s4}
we use this information to conclude with the proof of the claimed result and make some relevant comments.

\section{Arithmetic ingredients}\label{s2}

The notation $d_n$ will be used for the least common multiple of $1,2,\dots,n$. Recall that the prime number theorem is equivalent to the asymptotics
\begin{gather}
\lim_{n\to\infty}d_n^{1/n}=e.\label{e03}
\end{gather}

A rational function $S(t)$ of the form
\begin{gather*}
S(t)=\frac{P(t)}{(t-t_1)^{s_1}(t-t_2)^{s_2}\dotsb(t-t_q)^{s_q}},
\end{gather*}
whose denominator has degree larger than its numerator, possesses a unique partial-fraction decomposition
\begin{gather*}
S(t)=\sum_{j=1}^q\sum_{i=1}^{s_j}\frac{b_{i,j}}{(t-t_j)^i}.
\end{gather*}
The coefficients here can be computed on the basis of explicit formula
\begin{gather*}
b_{i,j}=\frac1{(s_j-i)!} \bigl(S(t)(t-t_j)^{s_j}\bigr)^{(s_j-i)}\Big|_{t=t_j}
\end{gather*}
for all $i$, $j$ in question. This procedure can be illustrated on the following examples, in which all the exponents $s_j$ are equal to~1:
\begin{gather*}
\frac{n!}{\prod_{j=0}^n(t+j)}
=\sum_{k=0}^n\frac{(-1)^k\binom nk}{t+k},
\displaybreak[2]\\
\frac{\prod_{j=1}^n(t-j)}{\prod_{j=0}^n(t+j)}
=\sum_{k=0}^n\frac{(-1)^{n+k}\binom{n+k}n\binom nk}{t+k},
\displaybreak[2]\\
\frac{\prod_{j=1}^n(t+n+j)}{\prod_{j=0}^n(t+j)}
=\sum_{k=0}^n\frac{(-1)^k\binom{2n-k}n\binom nk}{t+k},
\displaybreak[2]\\
\frac{2^{2n}\prod_{j=1}^n(t+\frac12-j)}{\prod_{j=0}^n(t+j)}
=\sum_{k=0}^n\frac{(-1)^{n+k}\binom{2n+2k}{2n}\binom{2n}{n+k}}{t+k},
\displaybreak[2]\\
\frac{2^{2n}\prod_{j=1}^n(t-\frac12+j)}{\prod_{j=0}^n(t+j)}
=\sum_{k=0}^n\frac{\binom{2k}k\binom{2n-2k}{n-k}}{t+k},
\displaybreak[2]\\
\frac{2^{2n}\prod_{j=1}^n(t+n-\frac12+j)}{\prod_{j=0}^n(t+j)}
=\sum_{k=0}^n\frac{(-1)^k\binom{4n-2k}{2n}\binom{2n}k}{t+k}.
\end{gather*}
It also means that the function $R(t)$ in \eqref{e02} can be written as
\begin{gather}
R(t)=\sum_{i=1}^s\sum_{k=0}^n\frac{a_{i,k}}{(t+k)^i}\label{e04}
\end{gather}
with the recipe to compute the coefficients $a_{i,k}$ in its partial-fraction decomposition. At the same time, the function $R(t)$ is a product of `simpler' rational functions given above, with all coefficients of their partial fractions being integral.

\begin{lemma}\label{l1}
Let $k_1,\dots,k_q$ be pairwise distinct numbers from the set $\{0,1,\dots,n\}$ and $s_1,\dots,s_q$ positive integers. Then the coefficients in the expansion
\begin{gather*}
\frac1{\prod_{j=1}^q(t+k_j)^{s_j}} =\sum_{j=1}^q\sum_{i=1}^{s_j}\frac{b_{i,j}}{(t+k_j)^i}
\end{gather*}
satisfy
\begin{gather} d_n^{s-i}b_{i,j}\in\mathbb Z, \qquad\text{where}\quad i=1,\dots,s_j \quad \text{and}\quad j=1,\dots,q, \label{e05}
\end{gather}
where $s=s_1+\dots+s_q$.

In particular,
\begin{gather}
d_n^{s-i}a_{i,k}\in\mathbb Z, \qquad\text{where}\quad i=1,\dots,s \quad \text{and}\quad k=0,1,\dots,n, \label{e06}
\end{gather}
for the coefficients in~\eqref{e04}.
\end{lemma}

\begin{proof} Denote the rational function in question by $S(t)$. The statement is trivially true when $q=1$, therefore we assume that $q\ge2$. In view of the symmetry of the data, it is sufficient to demonstrate the inclusions \eqref{e05} for $j=1$. Differentiating a related product $m$ times, for any $m\ge0$, we obtain
\begin{align*}
\frac1{m!}\bigl(S(t)(t+k_1)^{s_1}\bigr)^{(m)}
&=\frac1{m!}\left(\prod_{j=2}^q(t+k_j)^{-s_j}\right)^{(m)}
=\sum_{\substack{\ell_2,\dots,\ell_q\ge0\\ \ell_2+\dots+\ell_q=m}}\prod_{j=2}^q\frac1{\ell_j!}\bigl((t+k_j)^{-s_j}\bigr)^{(\ell_j)}\\
&=\sum_{\substack{\ell_2,\dots,\ell_q\ge0\\ \ell_2+\dots+\ell_q=m}}\prod_{j=2}^q(-1)^{\ell_j}\binom{s_j+\ell_j-1}{\ell_j}(t+k_j)^{-(s_j+\ell_j)}.
\end{align*}
This implies that
\begin{gather*}
b_{i,1}=\sum_{\substack{\ell_2,\dots,\ell_q\ge0\\ \ell_2+\dots+\ell_q=s_1-i}}
\prod_{j=2}^q(-1)^{\ell_j}\binom{s_j+\ell_j-1}{\ell_j}\frac1{(k_j-k_1)^{s_j+\ell_j}}
\end{gather*}
for $i=1,\dots,s_1$. Using $d_n/(k_j-k_1)\in\mathbb Z$ for $j=2,\dots,q$ and $\sum_{j=2}^q(s_j+\ell_j)=s-i$ for each individual summand, we deduce the desired inclusion in \eqref{e05} for $j=1$, hence for any~$j$.

The second claim in the lemma follows from considering $R(t)$ as a product of the `simpler' rational functions.
\end{proof}

\begin{lemma}\label{l2} For the coefficients $a_{i,k}$ in \eqref{e04}, we have
\begin{gather*}
a_{i,k}=(-1)^{i-1}a_{i,n-k} \qquad\text{for}\quad k=0,1,\dots,n \quad\text{and}\quad i=1,\dots,s,
\end{gather*}
so that
\begin{gather*}
\sum_{k=0}^na_{i,k}=0 \qquad\text{for} \ i \ \text{even}.
\end{gather*}
\end{lemma}

\begin{proof}
Since $s$ is odd, the function \eqref{e02} possesses the following (well-poised) symmetry: $R(-t-n)=-R(t)$. Substitution of the relation into~\eqref{e04} results in
\begin{align*}
-\sum_{i=1}^s\sum_{k=0}^n\frac{a_{i,k}}{(t+k)^i}
&=\sum_{i=1}^s\sum_{k=0}^n\frac{a_{i,k}}{(-t-n+k)^i}=\sum_{i=1}^s(-1)^i\sum_{k=0}^n\frac{a_{i,k}}{(t+n-k)^i}\\
&=\sum_{i=1}^s(-1)^i\sum_{k=0}^n\frac{a_{i,n-k}}{(t+k)^i},
\end{align*}
and the identities in the lemma follow from the uniqueness of decomposition into partial fractions. The second statement follows from
\begin{gather*}
\sum_{k=0}^na_{i,k}=(-1)^{i-1}\sum_{k=0}^na_{i,n-k}=(-1)^{i-1}\sum_{k=0}^na_{i,k}.\tag*{\qed}
\end{gather*}
\renewcommand{\qed}{}
\end{proof}

\begin{lemma}\label{l3} For each $n$,
\begin{gather*}
r_n=\sum_{\substack{i=2\\i\;\text{odd}}}^sa_i\zeta(i)+a_0
\qquad\text{and}\qquad
\hat r_n=\sum_{\substack{i=2\\i\;\text{odd}}}^sa_i\big(2^i-1\big)\zeta(i)+\hat a_0,
\end{gather*}
with the following inclusions available\textup:
\begin{gather*}
d_n^{s-i}a_i\in\mathbb Z \qquad\text{for}\quad i=3,5,\dots,s,
\quad\text{and}\quad d_n^sa_0, d_n^s\hat a_0\in\mathbb Z.
\end{gather*}
\end{lemma}

Notice that
\begin{gather*}
\big(2^i-1\big)\zeta(i) =\sum_{\ell=1}^\infty\frac1{\big(\ell-\frac12\big)^i}
\end{gather*}
for $i\ge2$.

\begin{proof}
Our strategy here is to write the series in \eqref{e01} using the partial-fraction decomposition~\eqref{e04} of $R(t)$. To treat the first sum $r_n$ we additionally introduce an auxiliary parameter $z>0$, which we later specialise to $z=1$:
\begin{align*}
r_n(z) &=\sum_{\nu=1}^\infty R_n(\nu)z^\nu
=\sum_{\nu=1}^\infty\sum_{i=1}^s\sum_{k=0}^n\frac{a_{i,k}z^\nu}{(\nu+k)^i}\\
&=\sum_{i=1}^s\sum_{k=0}^na_{i,k}z^{-k}\sum_{\nu=1}^\infty\frac{z^{\nu+k}}{(\nu+k)^i}
=\sum_{i=1}^s\sum_{k=0}^na_{i,k}z^{-k}\left(\Li_i(z)-\sum_{\ell=1}^k\frac{z^\ell}{\ell^i}\right)\\
&=\sum_{i=1}^s\Li_i(z)\sum_{k=0}^na_{i,k}z^{-k}-\sum_{i=1}^s\sum_{k=0}^n\sum_{\ell=1}^k\frac{a_{i,k}z^{-(k-\ell)}}{\ell^i},
\end{align*}
where
\begin{gather*}
\Li_i(z)=\sum_{\ell=1}^\infty\frac{z^\ell}{\ell^i}
\end{gather*}
for $i=1,\dots,s$ are the polylogarithmic functions. The latter are well defined at $z=1$ for $i\ge2$, where $\Li_i(1)=\zeta(i)$, while $\Li_1(z)=-\log(1-z)$ does not have a limit as $z\to1^-$. By taking the limit as $z\to1^-$ in the above derivation and using $R_n(\nu)=O\big(\nu^{-2}\big)$ as $\nu\to\infty$, we conclude that
\begin{gather*}
\sum_{k=0}^na_{1,k}=\lim_{z\to1^-}\sum_{k=0}^na_{1,k}z^{-k}=0,
\end{gather*}
and
\begin{gather}
r_n=\sum_{i=2}^s\zeta(i)\sum_{k=0}^na_{i,k}-\sum_{i=1}^s\sum_{k=0}^na_{i,k}\sum_{\ell=1}^k\frac1{\ell^i}.\label{e07}
\end{gather}

We proceed similarly for $\hat r_n$, omitting introduction of the auxiliary parameter $z$. Since $R(t)$ in \eqref{e02} vanishes at $t=-\frac12,-\frac32,\dots,-n+\frac12$, we can shift the starting point of summation for $\hat r_n$ to $t=-m-\frac12$, where $m=\big\lfloor\frac{n-1}2\big\rfloor$, so that
\begin{align}
\hat r_n &=\sum_{\nu=-m}^\infty R_n\big(\nu-\tfrac12\big) =\sum_{\nu=-m}^\infty\sum_{i=1}^s\sum_{k=0}^n\frac{a_{i,k}}{\big(\nu+k-\frac12\big)^i}
\nonumber\\
&=\sum_{i=1}^s\sum_{k=0}^na_{i,k}\sum_{\nu=-m}^\infty\frac1{\big(\nu+k-\frac12\big)^i} \displaybreak[2]\nonumber\\
&=\sum_{i=1}^s\sum_{k=0}^ma_{i,k}\sum_{\nu=-m}^\infty\frac1{\big(\nu+k-\frac12\big)^i}
+\sum_{i=1}^s\sum_{k=m+1}^na_{i,k}\sum_{\nu=-m}^\infty\frac1{\big(\nu+k-\frac12\big)^i} \displaybreak[2]\nonumber\\
&=\sum_{i=1}^s\sum_{k=0}^ma_{i,k}\left(\sum_{\ell=k-m}^0\frac1{\big(\ell-\frac12\big)^i}+\sum_{\ell=1}^\infty\frac1{\big(\ell-\frac12\big)^i}\right)
\nonumber\\ &\quad{} +\sum_{i=1}^s\sum_{k=m+1}^na_{i,k}\left(\sum_{\ell=1}^\infty\frac1{\big(\ell-\frac12\big)^i}-\sum_{\ell=1}^{k-m-1} \frac1{\big(\ell-\frac12\big)^i}\right)\displaybreak[2]\nonumber\\
&=\sum_{i=2}^s(2^i-1)\zeta(i)\sum_{k=0}^na_{i,k} +\sum_{i=1}^s\sum_{k=0}^ma_{i,k}\sum_{\ell=0}^{m-k}\frac{(-1)^i}{\big(\ell+\frac12\big)^i}
\nonumber\\ &\quad{} -\sum_{i=1}^s\sum_{k=m+1}^na_{i,k}\sum_{\ell=1}^{k-m-1}\frac1{\big(\ell-\frac12\big)^i}.\label{e08}
\end{align}
Now the statement of the lemma follows from the representations in \eqref{e07} and \eqref{e08}, Lemma~\ref{l2}, the inclusions \eqref{e06} of Lemma~\ref{l1} and
\begin{alignat*}{3}
& d_n^i\sum_{\ell=1}^k\frac1{\ell^i} \in\mathbb Z \qquad&& \text{for}\quad 0\le k\le n \quad \text{and}\quad i\ge1,&
\displaybreak[2]\\
& d_n^i\sum_{\ell=0}^{m-k}\frac{(-1)^i}{(\ell+\frac12)^i} \in\mathbb Z \qquad && \text{for}\quad 0\le k\le m \quad \text{and}\quad i\ge1,\\
& d_{n-1}^i\sum_{\ell=1}^{k-m-1}\frac1{\big(\ell-\frac12\big)^i} \in\mathbb Z \qquad && \text{for}\quad m+1\le k\le n \quad \text{and}\quad i\ge1.
\tag*{\qed}
\end{alignat*}\renewcommand{\qed}{}
\end{proof}

\section{Asymptotic behaviour}\label{s3}

In this section we make frequent use of Stirling's asymptotic formula
\begin{gather*}
n!\sim\sqrt{2\pi n}\left(\frac ne\right)^n \qquad\text{as}\quad n\to\infty,
\end{gather*}
and its corollary
\begin{gather*}
\binom{2n}n\sim\frac{2^{2n}}{\sqrt{\pi n}} \qquad\text{as}\quad n\to\infty
\end{gather*}
for the central binomial coefficients. (One may also use somewhat weaker but `more elementary' lower and upper bounds
\begin{gather*}
\int_1^n\log x\,\d x\le\log(n!)\le\int_2^{n+1}\log x\,\d x
\end{gather*}
for the factorial coming out from estimating integral sums of the logarithm function, with a~nemesis of running into more sophisticated versions for the asymptotics and inequalities below.)

Because the rational function $R_n(t)$ in \eqref{e02} vanishes at $1,2,\dots,n$ and at $\frac12,\frac32,\dots,\allowbreak n-\nobreak\frac12$,
the hypergeometric series \eqref{e01} can be alternatively written as
\begin{gather*}
r_n=\sum_{\nu=n+1}^\infty R_n(\nu)=\sum_{k=0}^\infty c_k \qquad\text{and}\qquad
\hat r_n=\sum_{\nu=n+1}^\infty R_n(\nu-\tfrac12)=\sum_{k=0}^\infty\hat c_k,
\end{gather*}
with the involved summands
\begin{gather}
c_k=R_n(n+1+k)=\frac{2^{6n}n!^{s-5}\prod_{j=0}^{6n}\big(k+1+\frac12j\big)}{\prod_{j=0}^n(n+k+1+j)^{s+1}}
=\frac{n!^{s-5}(6n+2k+2)! (n+k)!^{s+1}}{2 (2k+1)! (2n+k+1)!^{s+1}}\label{e09}
\\ \intertext{and}
\hat c_k=R_n\big(n+\tfrac12+k\big)=\frac{2^{6n}n!^{s-5}\prod_{j=0}^{6n}\big(k+\frac12+\frac12j\big)} {\prod_{j=0}^n\big(n+k+\frac12+j\big)^{s+1}}\nonumber
\end{gather}
strictly \emph{positive}. Observe that
\begin{align}
\frac{c_k}{\hat c_k}
&=\frac{\prod_{j=0}^{6n}(2k+2+j)}{\prod_{j=0}^{6n}(2k+1+j)}
\cdot\left(\prod_{j=0}^n\frac{n+k+\frac12+j}{n+k+1+j}\right)^{s+1}\nonumber\\
&=\frac{6n+2k+2}{2k+1}\cdot\left(2^{-2(n+1)}\frac{\binom{4n+2k+2}{2n+k+1}}{\binom{2n+2k}{n+k}}\right)^{s+1}\nonumber\\
&\sim\frac{6n+2k+2}{2k+1}\left(\frac{n+k}{2n+k+1}\right)^{(s+1)/2} \qquad\text{as}\quad n+k\to\infty. \label{e10}
\end{align}

\begin{lemma}\label{l4} For $s\ge7$ odd,
\begin{gather*}
\lim_{n\to\infty}r_n^{1/n}=\lim_{n\to\infty}\hat r_n^{1/n}=g(x_0)
\qquad\text{and}\qquad \lim_{n\to\infty}\frac{r_n}{\hat r_n}=1
\end{gather*}
where
\begin{gather*}
g(x)=\frac{2^6(x+3)^6(x+1)^{s+1}}{(x+2)^{2(s+1)}}
\end{gather*}
and $x_0$ is the unique positive zero of the polynomial
\begin{gather*}
x(x+2)^{(s+1)/2}-(x+3)(x+1)^{(s+1)/2}.
\end{gather*}
\end{lemma}

\begin{proof}We have
\begin{gather}
\frac{c_{k+1}}{c_k} =\frac{\big(k+3n+\frac32\big)(k+3n+2)}{(k+1)\big(k+\frac32\big)} \left(\frac{k+n+1}{k+2n+2}\right)^{s+1} \sim f\left(\frac kn\right)^2 \qquad\text{as}\quad n+k\to\infty, \label{e11}
\end{gather}
where
\begin{gather*}
f(x)=\frac{x+3}{x}\left(\frac{x+1}{x+2}\right)^{(s+1)/2}.
\end{gather*}
For an ease of notation write $q=(s+1)/2\ge4$. Since
\begin{gather*}
\frac{f'(x)}{f(x)}=\frac1{x+3}-\frac1x+q\left(\frac1{x+1}-\frac1{x+2}\right) =\frac{(q-3)x^2+3(q-3)x-6}{x(x+1)(x+2)(x+3)}
\end{gather*}
and the quadratic polynomial in the latter numerator has a unique positive zero $x_1$, the func\-tion~$f(x)$ monotone decreases from $+\infty$ to $f(x_1)$ when $x$ ranges from~$0$ to~$x_1$ and then monotone increases from $f(x_1)$ to $f(+\infty)=1$ (not attaining the value!) when $x$ ranges from $x_1$ to $+\infty$. In particular, there is exactly one positive solution $x_0$ to $f(x)=1$. Notice that $0<x_0<1$, because $f(1)=4\cdot(2/3)^q<1$.

The information gained and asymptotics in \eqref{e11} imply that $c_{k+1}/c_k>1$ for the indices $k<x_0n-\gamma\sqrt n$ and $c_{k+1}/c_k<1$
for $k>x_0n+\gamma\sqrt n$ for an appropriate choice of $\gamma>0$ dictated by application of Stirling's formula to the factorials defining $c_k$ in~\eqref{e09} (see \cite[Section~3.4]{Br58} as well as the second proof of Lemma~3 in~\cite{BR01}). This means that the asymptotic behaviour of the sum $r_n=\sum_{k=0}^\infty c_k$ is determined by the asymptotics of $c_{k_0}$ and its neighbours $c_k$, where $k_0=k_0(n)\sim x_0n$ and $|k-k_0|\le\gamma\sqrt n$, so that
\begin{align*}
\lim_{n\to\infty}r_n^{1/n}&=\lim_{n\to\infty}c_{k_0(n)}^{1/n}\\
&=\lim_{n\to\infty}\Biggl(\left(\frac ne\right)^{(s-5)n}\left(\frac{6n+2k_0+2}e\right)^{6n+2k_0+2}\left(\frac e{2k_0+1}\right)^{2k_0+1}
\\ &\qquad\quad\times
\left(\frac{n+k_0}e\right)^{(s+1)(n+k_0)}\left(\frac e{2n+k_0+1}\right)^{(s+1)(2n+k_0+1)}\Biggr)^{1/n}\\
&=\frac{(2x_0+6)^{2x_0+6}(x_0+1)^{(s+1)(x_0+1)}}{(2x_0)^{2x_0}(x_0+2)^{(s+1)(x_0+2)}}\\
&=\frac{2^6(x_0+3)^6(x_0+1)^{s+1}}{(x_0+2)^{2(s+1)}}\cdot f(x_0)^{2x_0}=g(x_0).
\end{align*}

It now follows from \eqref{e10} that
\begin{gather}
\frac{\hat c_{k+1}}{\hat c_k} \sim\frac{c_{k+1}}{c_k} \qquad\text{as}\quad n+k\to\infty, \label{e12}
\end{gather}
so that the above analysis applies to the sum $\hat r_n=\sum_{k=0}^\infty\hat c_k$ as well, and its asymptotic behaviour is determined by the asymptotics of $\hat c_{k_0}$ and its neighbours $\hat c_k$, where $k_0=k_0(n)\sim x_0n$ and $|k-k_0|\le\hat\gamma\sqrt n$. From \eqref{e12} we deduce that the limits of $\hat c_{k_0(n)}^{1/n}$ and $c_{k_0(n)}^{1/n}$ as $n\to\infty$ coincide, hence $\hat r_n^{1/n}\to g(x_0)$ as $n\to\infty$. In addition to this, we also get
\begin{gather*}
\lim_{n\to\infty}\frac{r_n}{\hat r_n}=\lim_{n\to\infty}\frac{c_{k_0(n)}}{\hat c_{k_0(n)}}
=\lim_{n\to\infty}\frac{6n+2k_0+2}{2k_0+1}\left(\frac{n+k_0}{2n+k_0+1}\right)^{(s+1)/2}=f(x_0),
\end{gather*}
which leads to the remaining limiting relation.
\end{proof}

\section{Conclusion}\label{s4}

We choose $s=25$ and apply Lemma~\ref{l4} to find out that $7r_n-\hat r_n>0 $ for $n$ sufficiently large, and
\begin{gather*}
\lim_{n\to\infty}(7r_n-\hat r_n)^{1/n}=g(x_0)=\exp(-25.292363\dots),
\end{gather*}
where $x_0=0.00036713\dots$. Assuming that the odd zeta values from $\zeta(5)$ to $\zeta(25)$ are all rational and denoting by $a$ their common denominator, we use Lemma~\ref{l3} and the asymptotics~\eqref{e03} to conclude that the sequence of \emph{positive integers}
\begin{gather*}
ad_n^{25}(7r_n-\hat r_n)
\end{gather*}
tends to $0$ as $n\to\infty$; contradiction. Thus, at least one of the numbers $\zeta(5),\zeta(7),\dots,\allowbreak\zeta(25)$ is irrational.

Those who count the prime number theorem as insufficiently elementary may use weaker versions of \eqref{e03}, for example, $d_n<3^n$ from \cite{Ha72} and the choice $s=33$ instead, to arrive at the same conclusion (for the larger value of $s$, of course).

Finally, we remark that the novelty of eliminating an `unwanted' term of $\zeta(3)$ in linear forms in odd zeta values can be further used with the arithmetic method in \cite{Zu04} to significantly reduce the size of~$s$. Since this does not let $s$ be down to $s=9$, hence leaving the achievement `at least one of the four numbers $\zeta(5)$, $\zeta(7)$, $\zeta(9)$, $\zeta(11)$ is irrational' unchanged, we do not discuss this generalisation in greater details. We point out, however, that there are other applications of the hypergeometric `twist-by-half' idea, some discussed in the joint papers~\cite{KZ18,RZ18}, and that a far-going extension to general `twists' introduced by J.~Sprang in \cite{Sp18} leads to an elementary proof of a version of the Ball--Rivoal theorem from~\cite{BR01} as well as to a significant improvement of the latter~--- see \cite{FSZ18} for details.

\subsection*{Acknowledgements}

I thank St\'ephane Fischler, Tanguy Rivoal, Johannes Sprang and the anonymous referees for their feedback on the manuscript.

\pdfbookmark[1]{References}{ref}
\LastPageEnding

\end{document}